\documentclass[12pt,reqno]{article}
\usepackage{amsfonts}
\usepackage{amsmath}
\usepackage{amsbsy}
\usepackage{graphicx}
\usepackage{amssymb,latexsym}
\usepackage{graphicx}
\usepackage{caption}
\usepackage{subcaption}
\numberwithin{equation}{section}
\newtheorem{theorem}{Theorem}

\newtheorem{corollary}[theorem]{Corollary}
\newtheorem{proposition}[subsection]{Proposition}
\newtheorem{example}[theorem]{Example}

\begin{document}

\title{Almost Contact Metric Structures Induced by $G_2$ Structures}
\begin{titlepage}
\author{N\"{u}lifer \"{O}zdemir\footnote{E.mail:
nozdemir@anadolu.edu.tr} \\
{\small Department of Mathematics, Anadolu University, 26470 Eski\c{s}ehir, Turkey} \\ \\
Mehmet Solgun\footnote{E.mail:
mehmet.solgun@bilecik.edu.tr}\\{\small Department of Mathematics,
Bilecik Seyh Edebali University, 11210 Bilecik, Turkey} \\ \\
\c{S}irin Aktay\footnote{E.mail:
sirins@anadolu.edu.tr} \\
{\small Department of Mathematics, Anadolu University, 26470 Eski\c{s}ehir, Turkey}}

\date{ }

\maketitle

\bigskip

\begin{abstract}
\noindent \noindent We study almost contact metric structures induced by 2-fold vector cross products on manifolds with $G_2$ structures. We get some results on possible classes of almost contact metric structures. Finally we give examples.
\vspace{1cm}

\noindent
\end{abstract}

\let\thefootnote\relax\footnotetext{This study was supported by Anadolu
University Scientific Research Projects Commission under the grant no: 1501F017.}

\end{titlepage}

\section{Introduction}
\newenvironment{proof}[1][Proof]{\noindent\textbf{#1} }{\ \rule{0.5em}{0.5em}} A recent research area in geometry is the relation between manifolds with structure group $G_2$ and almost contact metric manifolds. A manifold with $G_2$ structure has a 3-form globally defined on its tangent bundle with some properties. Such manifolds are classified into sixteen classes by Fern\'{a}ndez and Gray in \cite{FERNANDEZ} according to the properties of the covariant derivative of the 3-form.

On an almost contact metric manifold, there exists a global 2-form and the properties of the covariant derivative of this 2-form yields $2^{12}$ classes of almost contact metric manifolds, see \cite{ALEXIEV, CHINEA}.

Recently Matzeu and Munteanu constructed almost contact metric structure induced by the 2-fold vector cross product on some classes of manifolds with $G_2$ structures \cite{MATZEU}. Arikan et.al. proved the existence of almost contact metric structures on manifolds with $G_2$ structures \cite{ARIKAN}. Todd studied almost contact metric structures on manifolds with parallel $G_2$ structures \cite{TODD}.

Our aim is to study almost contact metric structures on manifolds with arbitrary $G_2$ structures. We eliminate some classes that almost contact metric structure induced from a $G_2$ structure may belong to according to properties of characteristic vector field of the almost contact metric structure. In particular, we also investigate the possible classes of almost contact metric structures on manifolds with nearly parallel $G_2$ structures. In addition, we give examples of almost contact metric structures on manifolds with $G_2$ structures induced by the 2-fold vector cross product.

\section{Preliminaries}

Consider $\mathbb{R}^7$ with the standard basis $\{e_1,..., e_7\}$. The fundamental 3-form on $\mathbb{R}^7$ is defined as
$$
\varphi_0=e^{123}+e^{145}+e^{167}+e^{246}-e^{257}-e^{347}-e^{356}
$$
where $\{e^1,..., e^7\}$ is the dual basis of the standard basis and $e^{ijk}=e^{i}\wedge e^{j}\wedge e^{k}$. Then compact, simple and simply connected 14-dimensional Lie group $G_2$ is
$$G_2:=\{ f\in GL(7, \mathbb{R}) \ | \ f^{*}\varphi_0=\varphi_0 \}.$$
A manifold with $G_2$ structure is a $7$-dimensional oriented manifold whose structure group reduces to the group $G_2$. In this case, there exists a global 3-form $\varphi$ on $M$ such that for all $p\in M$, $(T_pM,\varphi_p)\cong (\mathbb{R}^7,\varphi_0)$. This 3-form is called the fundamental 3-form or the $G_2$ structure on $M$ and gives a Riemannian metric $g$, a volume form and a 2-fold vector cross product $P$ on $M$ defined by $\varphi(x,y,z)=g(P(x,y),z)$ for all vector fields $x, y$ on $M$ \cite{BRY}.

Manifolds $(M,g)$ with $G_2$ structure $\varphi$ were classified according to properties of the covariant derivative of the fundamental 3-form. The space
$$\mathcal{W}=\{\alpha\in(\mathbb{R}^7)^*\otimes\Lambda^3(\mathbb{R}^7)^* | \alpha(x,y\wedge z\wedge P(y,z))=0\ \  \forall x,y,z\in \mathbb{R}^7\}$$ of of tensors having the same symmetry properties as the covariant derivative of $\varphi$ was written, and then this space was decomposed into four $G_2$-irreducible subspaces using the representation of the group $G_2$ on $\mathcal{W}$. Since
$$(\nabla\varphi)\!_p\!\!\in\!\! \mathcal{W}_p\!\!=\!\!\{\alpha\!\!\in\!\! T^*_pM\otimes\Lambda^3(T^*_pM) | \alpha(x,y\!\wedge z\!\wedge P(y,z))\!=\!0\ \  \forall x,y,z\!\!\in\!\! T_pM\}$$
and there are 16 invariant subspaces of $\mathcal{W}_p$, each subspace corresponds to a different class of manifolds with $G_2$ structure. For example, the class $\mathcal{P}$, in which the covariant derivative of $\varphi$ is zero, is the class of manifolds with parallel $G_2$ structure. A manifold which is in this class is sometimes called a $G_2$ manifold. $\mathcal{W}_1$ corresponds to the class of nearly parallel manifolds, which are manifolds with $G_2$ structure $\varphi$ satisfying $d\varphi=k\ast\varphi$ for some constant $k$ \cite{FERNANDEZ}.

Let $M^{2n+1}$ be a differentiable manifold of dimension $2n+1$. If there is a $(1,1)$ tensor field $\phi$, a vector field $\xi$ and a 1-form $\eta$ on $M$ satisfying
$$\phi^2=-I+\eta\otimes\xi, \qquad \eta(\xi)=1,$$
then $M$ is said to have an almost contact structure $(\phi, \xi, \eta)$. A manifold with an almost contact structure is called an almost contact manifold.

If in addition to an almost contact structure $(\phi, \xi, \eta)$, $M$ also admits a Riemannian metric $g$ such that
$$g(\phi(x),\phi(y))=g(x,y)-\eta(x)\eta(y)$$
for all vector fields $x, y$, then $M$ is an almost contact metric manifold with the almost contact metric structure $(\phi, \xi, \eta, g)$. The Riemannian metric $g$ is called a compatible metric. The 2-form $\Phi$ defined by $$\Phi(x,y)=g(x,\phi(y))$$ for all $x,y\in\Gamma(TM)$ is called the fundamental 2-form of the almost contact metric manifold $(M, \phi, \xi, \eta, g)$.

In \cite{CHINEA}, a classification of almost contact metric manifolds was obtained via the study of the covariant derivative of the fundamental 2-form. Let $(\xi,\eta,g)$ be an almost contact metric structure on $\mathbb{R}^{2n+1}$. A space
$$\begin{array}{rcl}\mathcal{C}&=&\left\{\alpha\in\otimes^0_3\mathbb{R}^{2n+1} | \alpha(x,y,z)=-\alpha(x,z,y)=-\alpha(x,\phi y,\phi z)\right.\\
\\
&&\left.\ \ \ \ \ \ \ \ \ \ \ \ \ \ \ \ \ \ \ \ \ +\eta(y)\alpha(x,\xi,z)+\eta(z)\alpha(x,y,\xi)\right\}
\end{array}$$
having the same symmetries as the covariant derivative of the fundamental 2-form was given. First this space was written as a direct sum of three subspaces
$$\mathcal{D}_1=\{\alpha\in\mathcal{C} | \alpha(\xi,x,y)=\alpha(x,\xi,y)=0\},$$
$$\mathcal{D}_2=\{\alpha\in\mathcal{C} | \alpha(x,y,z)=\eta(x)\alpha(\xi,y,z)+\eta(y)\alpha(x,\xi,z)+\eta(z)\alpha(x,y,\xi)\}$$
and
$$\mathcal{C}_{12}=\{\alpha\in\mathcal{C} | \alpha(x,y,z)=\eta(x)\eta(y)\alpha(\xi,\xi,z)+\eta(x)\eta(z)\alpha(\xi,y,\xi)\}$$ and then, $\mathcal{D}_1$, $\mathcal{D}_2$ were decomposed into $U(n)\times 1$ irreducible components $\mathcal{C}_1,\ldots,\mathcal{C}_{4}$ and $\mathcal{C}_5,\ldots,\mathcal{C}_{11}$, respectively. Thus there are $2^{12}$ invariant subspaces, denoted by $\mathcal{C}_1,\ldots,\mathcal{C}_{12}$, each corresponding to a class of almost contact metric manifolds. For example, the trivial class such that $\nabla\Phi=0$ corresponds to the class of cosymplectic \cite{BLAIR} (called co-K{\"a}hler by some authors) manifolds, $\mathcal{C}_1$ is the class of nearly-K-cosymplectic manifolds, etc.

In the classification of Chinea and Gonzales, it was shown that the space of quadratic invariants of $\mathcal{C}$ is generated by the following 18 elements:
\\

 \begin{tabular}{ll}
 $i_1(\alpha) =  \sum\limits_{i,j,k}\alpha(e_i,e_j,e_k)^2$ & $i_2(\alpha)=\sum\limits_{i,j,k}\alpha(e_i,e_j,e_k)\alpha(e_j,e_i,e_k)$ \\
 $i_3(\alpha) =  \sum\limits_{i,j,k}\alpha(e_i,e_j,e_k)\alpha(\phi e_i,\phi e_j, e_k)$ & $i_4(\alpha)=\sum\limits_{i,j,k}\alpha(e_i,e_i,e_k)\alpha(e_j,e_j,e_k)$ \\
 $i_5(\alpha) =  \sum\limits_{j,k}\alpha(\xi,e_j,e_k)^2$ & $i_6(\alpha) =  \sum\limits_{i,k}\alpha(e_i,\xi,e_k)^2$ \\
$i_7(\alpha)=\sum\limits_{j,k}\alpha(\xi,e_j,e_k)\alpha(e_j,\xi,e_k)$ & $i_8(\alpha)=\sum\limits_{i,j}\alpha(e_i,e_j,\xi)\alpha(e_j,e_i,\xi)$ \\
$i_9(\alpha) =  \sum\limits_{i,j}\alpha(e_i,e_j,\xi)\alpha(\phi e_i,\phi e_j, \xi)$ & $i_{10}(\alpha)=\sum\limits_{i,j}\alpha(e_i,e_i,\xi)\alpha(e_j,e_j,\xi)$ \\
$i_{11}(\alpha) =  \sum\limits_{i,j}\alpha(e_i,e_j,\xi)\alpha(e_j,\phi e_i,\xi)$ & $i_{12}(\alpha) =  \sum\limits_{i,j}\alpha(e_i,e_j,\xi)\alpha(\phi e_j,\phi e_i,\xi)$ \\
 \end{tabular}
 \begin{tabular}{ll}
$i_{13}(\alpha)=\sum\limits_{j,k}\alpha(\xi,e_j,e_k)\alpha(\phi e_j,\xi,e_k)$ & $i_{14}(\alpha)=\sum\limits_{i,j}\alpha(e_i,\phi e_i,\xi)\alpha(e_j,\phi e_j,\xi)$ \\
 $i_{15}(\alpha)=\sum\limits_{i,j}\alpha(e_i,\phi e_i,\xi)\alpha(e_j,e_j,\xi)$ & $i_{16}(\alpha) =  \sum\limits_{k}\alpha(\xi,\xi,e_k)^2$ \\
 $i_{17}(\alpha) =  \sum\limits_{i,k}\alpha(e_i,e_i,e_k)\alpha(\xi,\xi,e_k)$ &  $i_{18}(\alpha) =  \sum\limits_{i,k}\alpha(e_i,e_i,\phi e_k)\alpha(\xi,\xi,e_k)$
 \end{tabular}
\\
 where $\lbrace e_1,e_2,...,e_6,\xi \rbrace $ is a local orthonormal basis. Also following relations among quadratic invariants were expressed for manifolds having dimensions $\geq 7$, where $\alpha\in\mathcal{C}$ and $A=\{1,2,3,4,5,7,11,13,15,16,17,18\}$:
\\
$\mathbf{\mathcal{C}_1:} i_1(\alpha)=-i_2(\alpha)=-i_3(\alpha)=||\alpha ||^2;  \ \ i_m(\alpha)=0\ (m\geq 4)$\\
$\mathbf{\mathcal{C}_2:} i_1(\alpha)=2i_2(\alpha)=-i_3(\alpha)=||\alpha ||^2; \quad i_m(\alpha)=0\ (m\geq 4)$\\
$\mathbf{\mathcal{C}_3:} i_1(\alpha)=i_3(\alpha)=||\alpha ||^2; \quad i_2(\alpha)=i_m(\alpha)=0\ (m\geq 4)$\\
$\mathbf{\mathcal{C}_4:} i_1(\alpha)=i_3(\alpha)=\frac{n}{(n-1)^2}i_4(\alpha)=\frac{n}{(n-1)^2}\sum\limits_k^{2n}c_{12}^2(\alpha)(e_k)$;\\
$i_2(\alpha)=i_m(\alpha)=0\ (m> 4) $\\
$\mathbf{\mathcal{C}_5:} i_6(\alpha)=-i_8(\alpha)=i_9(\alpha)=-i_{12}(\alpha)=\frac{1}{2n}i_{14}(\alpha)$;\\
$i_{10}(\alpha)=i_m(\alpha)=0 \ \  (m\in A) $\\
$\mathbf{\mathcal{C}_6:} i_6(\alpha)=i_8(\alpha)=i_9(\alpha)=i_{12}(\alpha)=\frac{1}{2n}i_{10}(\alpha)$;\\
$i_{14}(\alpha)=i_m(\alpha)=0 \ \ (m\in A)  $\\
$\mathbf{\mathcal{C}_7:} i_6(\alpha)=i_8(\alpha)=i_9(\alpha)=-i_{12}(\alpha)=\frac{||\alpha ||^2}{2}$;\\
$i_{10}(\alpha)=i_{14}(\alpha)=i_m(\alpha)=0 \ \ (m\in A)  $\\
$\mathbf{\mathcal{C}_8:} i_6(\alpha)=-i_8(\alpha)=i_9(\alpha)=-i_{12}(\alpha)=\frac{||\alpha ||^2}{2}$;\\
$i_{10}(\alpha)=i_{14}(\alpha)=i_m(\alpha)=0\ \ (m\in A) $\\
$\mathbf{\mathcal{C}_9:} i_6(\alpha)=i_8(\alpha)=-i_9(\alpha)=-i_{12}(\alpha)=\frac{||\alpha ||^2}{2}$;\\
$i_{10}(\alpha)=i_{14}(\alpha)=i_m(\alpha)=0\ \  (m\in A) $ \\
$\mathbf{\mathcal{C}_{10}:} i_6(\alpha)=-i_8(\alpha)=-i_9(\alpha)=i_{12}(\alpha)=\frac{||\alpha ||^2}{2}$;\\  $i_{10}(\alpha)=i_{14}(\alpha)=i_m(\alpha)=0 \ \ (m\in A)$ \\
$\mathbf{\mathcal{C}_{11}:} i_5(\alpha)=||\alpha||^2;  i_m(\alpha)=0\ \ (m\neq 5)$ \\
$\mathbf{\mathcal{C}_{12}:} i_{16}(\alpha)=||\alpha||^2;  i_m(\alpha)=0\ \ (m\neq 16)$ \\

For details, refer to \cite{CHINEA}.

We give below most studied classes of almost contact metric structures as direct sum of spaces $\mathcal{C}_i$:
$$\mid\mathcal{C}\mid=\text{the class of cosymplectic manifolds.}$$
$$\mathcal{C}_1=\text{the class of nearly-K-cosymplectic manifolds.}$$
$$\mathcal{C}_2\oplus\mathcal{C}_9=\text{the class of almost cosymplectic manifolds.}$$
$$\mathcal{C}_5=\text{the class of $\alpha$-Kenmotsu manifolds.}$$
$$\mathcal{C}_6=\text{the class of $\alpha$-Sasakian manifolds.}$$
$$\mathcal{C}_5\oplus\mathcal{C}_6=\text{the class of trans-Sasakian manifolds.}$$
$$\mathcal{C}_6\oplus\mathcal{C}_7=\text{the class of quasi-Sasakian manifolds.}$$
$$\mathcal{C}_3\oplus\mathcal{C}_7\oplus\mathcal{C}_8=\text{the class of semi-cosymplectic and normal manifolds.}$$
$$\mathcal{C}_1\oplus\mathcal{C}_5\oplus\mathcal{C}_6=\text{the class of nearly trans-Sasakian manifolds.}$$
$$\mathcal{C}_1\oplus\mathcal{C}_2\oplus\mathcal{C}_9\oplus\mathcal{C}_{10}=\text{the class of quasi-K-cosymplectic manifolds.}$$
$$\mathcal{C}_3\oplus\mathcal{C}_4\oplus\mathcal{C}_5\oplus\mathcal{C}_6\oplus\mathcal{C}_7\oplus\mathcal{C}_8=\text{the class of normal manifolds.}$$
$$\mathcal{D}_1\oplus\mathcal{C}_5\oplus\mathcal{C}_6\oplus\mathcal{C}_7\oplus\mathcal{C}_8\oplus\mathcal{C}_9\oplus\mathcal{C}_{10}=\text{the class of almost-K-contact manifolds.}$$
$$\mathcal{C}_1\oplus\mathcal{C}_2\oplus\mathcal{C}_3\oplus\mathcal{C}_7\oplus\mathcal{C}_8\oplus\mathcal{C}_9\oplus\mathcal{C}_{10}\oplus\mathcal{C}_{11}=\text{the class of semi-cosymplectic manifolds.}$$

Let $(M, g)$ be a 7-dimensional Riemannian manifold with $G_2$ structure $\varphi$ and the associated 2-fold vector cross product $\times$ and let $\xi$ be a nowhere zero vector field of unit length on $M$. Then for
$$\phi(x):=\xi\times x \qquad \eta(x):=g(\xi, x),$$
$(\phi, \xi, \eta, g)$ is an almost contact metric structure on $M$ \cite{MATZEU, ARIKAN}. Throughout this study, $(\phi, \xi, \eta, g)$ will denote the almost contact metric structure (a.c.m.s) induced by the $G_2$ structure $\varphi$ on $M$ and $\Phi$ the fundamental 2-form of the a.c.m.s.

\section{Almost contact metric structures obtained from $G_2$ Structures}

Let $M$ be a manifold with $G_2$ structure $ \varphi$ and $\xi$ a nowhere zero unit vector field on $M$ and $(\phi, \xi, \eta, g)$ the a.c.m.s. with the fundamental form $\Phi$ induced by the $G_2$ structure $\varphi$.

If $\nabla \varphi=0$, then
it can be seen that $\nabla \Phi =0$ if and only if $\nabla \xi=0$ \cite{COSYM, TODD}.

If $\xi$ is a Killing vector field on a manifold with any $G_2$ structure, then
$$
\begin{array}{rcl}
d\eta(x,y)&=&\frac{1}{2}\{(\nabla_x\eta)(y)-(\nabla_y\eta)(x)\}\\
&=&\frac{1}{2}\{g(\nabla_x\xi,y)-g(\nabla_y\xi,x)\}\\
&=&g(\nabla_x\xi,y),
\end{array}
$$
which implies $$d\eta=0 \Leftrightarrow \nabla\xi=0.$$

Therefore if the Killing vector field $\xi$ is not parallel, then the a.c.m.s. can not be nearly-K-cosymplectic ($\mathcal{C}_1$).

To deduce further results, we focus on the covariant derivative of the fundamental 2-form $\Phi$, where the a.c.m.s. $(\phi, \xi, \eta, g)$ is obtained from a $G_2$ structure of any class and $\xi$ is any nonzero vector field. Direct calculation gives
\begin{equation}
(\nabla_x\Phi)(y,z)=g(y,\nabla_x(\xi\times z))+g(\nabla_xz,\xi\times y).
\end{equation}
We also compute some of $i_k(\nabla\Phi),(k=1,...,18)$ to understand which class $\nabla \Phi$ may belong to.

\begin{proposition} \label{proposition1}
Let $\varphi$ be  a $G_2$ structure on $M$ of an arbitrary class and $(\phi, \xi, \eta, g)$ an a.c.m.s. obtained from $\varphi$.  Then
\begin{itemize}
\item[a.]
$i_6(\nabla\Phi)=0$ if and only if $\nabla_{e_i} \xi =0$ for $i=1,\cdots, 6$ (Note that $\nabla_{\xi} \xi $ need not be zero),
\item[b.]
$i_{16}(\nabla\Phi)=0$ if and only if $\nabla_\xi\xi=0 $.
\end{itemize}
\end{proposition}
\begin{proof}
By direct calculation, for any $i,k\in \lbrace1,2,...,6\rbrace$
\begin{align*}
(\nabla_{e_i}\Phi)(\xi,e_k)&=g(\xi,\nabla_{e_i}(\xi\times e_k))+g(\nabla_{e_i}e_k,\xi\times\xi)\\
&=g(\xi,\nabla_{e_i}(\xi \times e_k)) \\
&=-g(\nabla_{e_i}\xi,\xi \times e_k )
\end{align*}
and thus, we obtain
$$i_6(\nabla\Phi) = \sum\limits_{i,k}((\nabla_{e_i}\Phi)(\xi,e_k))^2=\sum\limits_{i,k} g(\nabla_{e_i}\xi,\xi \times e_k )^2. $$
Since $\xi \times e_k$ is also a frame element,
$i_6(\nabla\Phi)=0$ if and only if $\nabla_{e_i} \xi $ is zero.

Similarly,
   \begin{align*}
  (\nabla_{\xi}\Phi)(\xi,e_k)&=g(\xi,\nabla_\xi(\xi\times e_k))+g(\nabla_\xi e_k,\xi\times\xi) \\
  &=-g(\nabla_\xi\xi,\xi\times e_k)
  \end{align*}
for any $k\in \lbrace1,2,...,6\rbrace$, and we get
$$i_{16}(\nabla\Phi) =  \sum\limits_{k}(\nabla_{\xi}\Phi)(\xi,e_k)^2=\sum\limits_{k} g(\nabla_\xi\xi,\xi\times e_k)^2.$$
Note that $g(\nabla_{\xi}\xi,\xi)=0$ since $\xi$ is of unit length.
As a result, $i_{16}(\nabla\Phi)=0$ if and only if $\nabla_\xi\xi=0.$
\end{proof}

\begin{proposition}
 \label{proposition2}
 Let $(\phi,\eta,\xi,g)$ be an almost contact metric structure induced by a $G_2$ structure $\varphi$. Then,
\begin{itemize}
\item
$i_{14}(\nabla\Phi)=0$ if and only if $div (\xi)=0.$
\item
$i_{15}(\nabla\Phi)=-div(\xi)g(\xi,v), $
where $v=\sum\limits_{j=1}^6 e_j\times (\nabla_{e_j}\xi)$.
\end{itemize}
\end{proposition}

\begin{proof}
For any $i,j\in \lbrace1,2,...,6\rbrace$ we have

\begin{align*}
(\nabla_{e_i}\Phi)(\phi e_i,\xi)&=g(\xi\times e_i,\nabla_{e_i}(\xi\times\xi))+g(\nabla_{e_i}\xi,\xi\times(\xi\times e_i)) \\
&=-g(\nabla_{e_i}\xi,e_i) \\
&=g(\xi,\nabla_{e_i}e_i).
\end{align*}

On the other hand,
\begin{align*}
\sum_{i=1}^6\nabla_{e_i}e_i&=-\sum_{i=1}^6 div(e_i)e_i-div(\xi)\xi-\nabla_\xi\xi
\end{align*}
and thus
\begin{align*}
g(\xi,\sum\limits_i\nabla_{e_i}e_i)&=-g(\xi,\sum\limits_i div(e_i)e_i)-g(\xi,div(\xi)\xi)
-g(\xi,\nabla_\xi\xi) \\
&=-div(\xi).
\end{align*}

Then
$$
\begin{array}{rcl}
i_{14}(\nabla\Phi)&=&
\sum\limits_{i,j}(\nabla_{e_i}\Phi)(\phi e_i,\xi)(\nabla_{e_j}\Phi)( \phi e_j,\xi)\\
&=& \Big(g(\xi,\sum\limits_i\nabla_{e_i}e_i)\Big) \Big( g(\xi,\sum\limits_j \nabla_{e_j}e_j)\Big)=(div(\xi))^2.
\end{array}$$

Therefore, $i_{14}(\nabla\Phi)$ is zero if and only if $div(\xi)$ is zero.

Similarly, from equations
 $$(\nabla_{e_i}\Phi)(\phi e_i,\xi)=-g(\nabla_{e_i}\xi,e_i) \ \mbox{and}  \ (\nabla_{e_j}\Phi)(e_j,\xi)=g(\nabla_{e_j}\xi,\xi\times e_j)$$
 we have,
\begin{align*}
 i_{15}(\nabla\Phi)&=\sum\limits_{i,j}(\nabla_{e_i}\Phi) (\phi e_i,\xi) \ (\nabla_{e_j}\Phi)( e_j,\xi)\\
 &=\sum\limits_{i,j}g(\xi,\nabla_{e_i}e_i)g(\nabla_{e_j}\xi,\xi\times e_j) \\
&=\Big(g(\xi,\sum\limits_i \nabla_{e_i}e_i)\Big)\Big(\sum\limits_jg(\xi,\nabla_{e_j}(e_j\times\xi))\Big)\\
&=\Big(g(\xi,-div(\xi)\xi)-g(\xi,\sum\limits_i div(e_i)e_i)\Big)\Big( \sum\limits_j g(\xi,e_j \times\nabla_{e_j} \xi)\Big) \\
&=-div(\xi).g(\xi,v).
\end{align*}

\end{proof}

Now consider in particular an a.c.m.s. induced by a nearly parallel $G_2$ structure.
\begin{proposition}  \label{proposition3}
Let $(\phi,\eta,\xi,g)$ be an almost contact metric structure induced by a nearly parallel $G_2$ structure. Then,
\begin{itemize}
\item
$i_{5}(\nabla\Phi)=0 $ if and only if $\nabla_{\xi} \xi =0$.
\item
If $\nabla_{\xi} \xi =0$, then  $i_{17}(\nabla\Phi)=i_{18}(\nabla\Phi)=0$.
\end{itemize}
\end{proposition}

\begin{proof}
Since $\varphi$ is nearly parallel, for any $j,k\in \lbrace1,2,...,6\rbrace$ we have

\begin{align*}
(\nabla_\xi\Phi)(e_j,e_k)&=g(e_j,\nabla_\xi(\xi\times e_k))+g(\nabla_\xi e_k,\xi\times e_j)\\
&=g(e_j,\nabla_\xi\xi\times e_k)+g(e_j,\xi\times\nabla_\xi e_k)+g(\nabla_\xi e_k,\xi\times e_j) \\
&=-g(\nabla_\xi\xi,e_j\times e_k).
\end{align*}
So,\\
 $$i_5(\nabla\Phi)=\sum\limits_{j,k}((\nabla_\xi\Phi)(e_j,e_k))^2=\sum\limits_{j,k}(g(\nabla_\xi\xi,e_j\times e_k))^2$$
 which is zero if and only if $\nabla_\xi \xi $ is zero. Here, $e_j\times e_k$ is also a frame element.

Similarly,
 For any $i,k\in \lbrace1,2,...,6\rbrace$,
 \begin{align*}
 (\nabla_{e_i}\Phi)(e_i,\phi e_k)&=g(e_i,\nabla_{e_i}(\xi\times(\xi\times e_k))+g(\nabla_{e_i}(\xi\times e_k),\xi\times e_i) \\
 &=g(e_i,\nabla_{e_i}(-e_k))+g(\nabla_{e_i}(\xi\times e_k),\xi\times e_i) \\
 &=g(\nabla_{e_i}e_i,e_k)+g(\nabla_{e_i}(\xi\times e_k),\xi\times e_i)
 \end{align*}
\begin{align*}
(\nabla_{\xi}\Phi)(\xi,e_k)&=g(\xi,\nabla_\xi(\xi\times e_k))+g(\nabla_\xi e_k,\xi\times\xi)\\
&=g(\xi,\nabla_\xi \xi\times e_k)+g(\xi,\xi\times\nabla_\xi e_k) \\
&=-g(e_k,(\nabla_\xi \xi)\times\xi).
\end{align*}
Then
\begin{align*}
i_{18}(\nabla\Phi) &=\sum\limits_{i,k}((\nabla_{e_i}\Phi)(e_i,\phi e_k))((\nabla_{\xi}\Phi)(\xi,e_k))\\
&=-\sum\limits_{i,k}\Big(g(\nabla_{e_i}e_i,e_k)+g(\nabla_{e_i}(\xi\times e_k),\xi\times e_i) \Big) \Big(g(e_k,(\nabla_\xi \xi)\times\xi) \Big) \\
&=-\sum\limits_{i,k}\Big( g(\nabla_{e_i}e_i,e_k)g(e_k,(\nabla_\xi \xi)\times\xi) \Big) \\
&\quad -\sum\limits_{i,k} \Big( g(\nabla_{e_i}(\xi\times e_k),\xi\times e_i)g(e_k,(\nabla_\xi \xi)\times\xi) \Big) \\
&=-\sum\limits_{i,k}\Big( g(\nabla_{e_i}e_i,e_k)g(e_k,(\nabla_\xi \xi)\times\xi) \Big) \\
& \quad +\sum\limits_{i,k}\Big( g(\nabla_{e_i}e_i,e_k)g(e_k,(\nabla_\xi \xi)\times\xi) \Big) \\
& \quad -\sum\limits_{i,k}\Big( g(\xi\times e_k,e_i\times \nabla_{e_i}\xi)g(e_k,(\nabla_\xi \xi)\times\xi \Big) \\
&=-\sum\limits_i g(\xi\times (\sum\limits_k g((\nabla_\xi\xi)\times \xi,e_k)e_k+g((\nabla_\xi\xi)\times\xi,\xi)\xi),e_i \times \nabla_{e_i}\xi) \\
&=-\sum\limits _i g(\xi\times((\nabla_\xi\xi)\times\xi),e_i\times\nabla_{e_i}\xi) \\
&=-g(\nabla_\xi\xi,\sum\limits_i (e_i\times\nabla_{e_i}\xi)).
\end{align*}
Thus, if $\nabla_\xi\xi$ is zero, so is $i_{18}(\nabla\Phi)$.

For $i_{17}$ we compute
\begin{equation*}
(\nabla_{e_i}\Phi)(e_i,e_k)=g(e_i,\nabla_{e_i}(\xi\times e_k))+g(\nabla_{e_i}e_k,\xi\times e_i)
\end{equation*}
and
\begin{align*}
 (\nabla_\xi\Phi)(\xi,e_k)&=g(\xi,\nabla_\xi(\xi\times e_k))+g(\nabla_\xi e_k,\xi\times\xi) \\
 &=-g(\nabla_\xi\xi,\xi\times e_k) \\
 &=g(e_k,\xi\times(\nabla_\xi\xi))
 \end{align*}
for any $i,k\in \lbrace1,2,...,6\rbrace$ and obtain
 \begin{align*}
 i_{17}(\nabla\Phi)&=\sum\limits_{i,k}((\nabla_{e_i}\Phi)(e_i,e_k))((\nabla_{\xi}\Phi)(\xi,e_k))\\
 &=  \sum\limits_{i,k}\Big( -g(\nabla_{e_i}e_i,\xi\times e_k)-g(e_k,\nabla_{e_i}(\xi\times e_i)\Big)\Big(g(e_k,\xi\times(\nabla_\xi\xi) \Big) \\
 &= \sum\limits_{i,k}g(e_k,\xi\times(\nabla_{e_i}e_i)g(e_k,\xi\times\nabla_\xi\xi)- \sum\limits_{i,k}g(e_k,\xi\times(\nabla_{e_i}e_i)g(e_k,\xi\times\nabla_\xi\xi)\\
 & +\sum\limits_{i,k}g(e_k,e_i \times\nabla_{e_i}\xi)g(e_k,\xi\times\nabla_\xi\xi) \\
 &=\sum\limits_{i,k}g(e_k,e_i \times\nabla_{e_i}\xi)g(e_k,\xi\times\nabla_\xi\xi) \\
 &=g(\xi\times(\nabla_\xi\xi),\sum\limits_i e_i\times(\nabla_{e_i}\xi))
 \end{align*}
 Thus, if $ \nabla_\xi\xi=0$, then $i_{17}(\nabla\Phi)=0$.
 \end{proof}

Similarly, if $\nabla\xi$ is zero, then  so is $i_{15}(\nabla\Phi)$.

Before giving results on possible classes of a.c.m.s. induced by $G_2$ structures, note that $\delta\eta=-div(\xi)$.
To see this, consider the orthonormal basis $ \{ e_1, \cdots, e_6, \xi\}$. Then

 \begin{align*}
 div(\xi)&=\sum\limits_{i=1}^6g(\nabla_{e_i}\xi,e_i)+g(\nabla_\xi\xi,\xi) \\
 &=\sum\limits_{i=1}^6g(\nabla_{e_i}\xi,e_i).
\end{align*}

On the other hand, since

 \begin{align*}
 (\nabla_{e_i}\eta)(e_i)&=e_i[\eta(e_i)]-\eta(\nabla_{e_i}e_i) \\
 &=g(\nabla_{e_i}\xi,e_i)+g(\xi,\nabla_{e_i}e_i)-g(\xi,\nabla_{e_i}e_i) \\
 &=g(\nabla_{e_i}\xi,e_i),
 \end{align*}

we have
$$\delta\eta=-\sum\limits_{i=1}^6(\nabla_{e_i}\eta)(e_i)=-\sum\limits_{i=1}^6g(\nabla_{e_i}\xi,e_i)=-div(\xi).$$

\begin{theorem} \label{theorem1}
Let $M$ be a manifold with a $G_2$ structure $\varphi$ and $(\phi,\xi,\eta,g)$ be an almost contact metric structure (a.c.m.s.) obtained from $\varphi$.\\
(a) If $\nabla_{\xi}\xi\neq 0$, then $\nabla\Phi$ can not be in classes $\mathcal{D}_2, \mathcal{C}_1,\mathcal{C}_2,\cdots,\mathcal{C}_{11}$.\\
(b) If $div(\xi)\neq 0$, then the almost contact metric structure can not belong to classes $\mathcal{D}_1 $, $\mathcal{C}_i$ for $i=1,2,3,4,6,7,\cdots,12$ and can not be semi-cosymplectic ($\mathcal{C}_1 \oplus \mathcal{C}_2\oplus \mathcal{C}_3\oplus \mathcal{C}_7\oplus \mathcal{C}_8\oplus \mathcal{C}_9\oplus \mathcal{C}_{10}\oplus \mathcal{C}_{11}$).
\end{theorem}

In following proofs we use the relations below given in \cite{CHINEA} together with properties of $i_m$ for each $\mathcal{C}_i$:\\
If $\alpha\in\mathcal{D}_1$, then $i_m(\alpha)=0$ for $m\geq 5$.\\
If $\alpha\in\mathcal{D}_2$, then $i_m(\alpha)=0$ for $m=1,2,3,4,16,17,18$.\\

\begin{proof}
(a) Let $\nabla_{\xi}\xi\neq0$. Then by the proposition [\ref{proposition1}], we have $i_{16}(\nabla\Phi)\neq0$. This implies $\nabla\Phi\notin \mathcal{D}_2$. In addition, $\nabla\Phi$ can not belong to any of the classes $C_i$, $i=1,\ldots,11$.

(b)If $div(\xi)\neq 0$, then the proposition [\ref{proposition2}] yields that $i_{14}(\nabla\Phi)=(div(\xi))^2\neq 0$. Hence $\nabla\Phi$ can not satisfy the defining relations of the classes
$$\mathcal{D}_1= \mathcal{C}_1\oplus \mathcal{C}_2\oplus \mathcal{C}_3 \oplus \mathcal{C}_4,$$
$$\mathcal{C}_1,\mathcal{C}_2,\mathcal{C}_3,\mathcal{C}_4,\mathcal{C}_6,\cdots ,\mathcal{C}_{12}. $$
Besides, the defining relation of semi cosymplectic manifolds is
$$ \delta \Phi=0 \ \mbox{and} \ \delta\eta=0.$$
Since $div(\xi)\neq 0$, $\delta\eta \neq 0$, and thus the a.c.m.s. is not semi cosymplectic.
\end{proof}

Note that if $\nabla_{\xi}\xi\neq 0$, then since $\nabla\Phi\notin\mathcal{D}_2=\mathcal{C}_5\oplus\ldots\oplus\mathcal{C}_{11}$, the a.c.m.s. can not be contained in any subclass of $\mathcal{D}_2$. In particular, the a.c.m.s. can not be $\alpha$-Kenmotsu, $\alpha$-Sasakian, trans-Sasakian or quasi-Sasakian.

If $div(\xi)\neq 0$, then we have $\nabla\Phi\notin\mathcal{D}_1=\mathcal{C}_1\oplus\ldots\oplus\mathcal{C}_4$. In this case, the a.c.m.s. can not be nearly-K-cosymplectic. Also, since the a.c.m.s. can not be semi-cosymplectic, it can not be almost-cosymplectic, $\alpha$-Kenmotsu, $\alpha$-Sasakian, trans-Sasakian, normal semi-cosymplectic or quasi-K-cosymplectic.

Note also that the class $\mathcal{C}_{12}$ is not contained in the class of semi-cosymplectic manifolds. We give a proof together with examples in \cite{Hermitian}. For completeness, we also remind the proof here.

The defining relation of $\mathcal{C}_{12}$ gives $$\nabla_{\xi}\Phi(\xi,x)=-(\nabla_{\xi}\eta)(\phi(x)),$$ which does not have to be zero. Assume that $$(\nabla_{\xi}\eta)(\phi(x))=0$$ for all vector fields $x$. Then replacing $x$ with $\phi(x)$, we get $(\nabla_{\xi}\eta)(x)=0$, but then the defining relation of $\mathcal{C}_{12}$ implies $\nabla_x\Phi=0$ for all $x$, that is there is no element in $\mathcal{C}_{12}$ which is not in the trivial class. Thus there is a vector field $x_0$ on $M$ such that
$$(\delta\Phi)(x_0)=-\nabla_{\xi}\Phi(\xi,x_0)=(\nabla_{\xi}\eta)(\phi(x_0))\neq 0.$$ Therefore the class $\mathcal{C}_{12}$ is not semi-cosymplectic.

Consider an a.c.m.s. induced by a nearly parallel $G_2$ structure. We deduce following results.

\begin{theorem} \label{theorem2}
Let $(\phi,\xi,\eta,g)$ be an a.c.m.s. obtained from a nearly parallel $G_2$ structure $\varphi$. If $\nabla_\xi\xi \neq0$, then $\nabla\Phi$ can not be in classes $\mathcal{D}_1,\mathcal{D}_2,\mathcal{C}_{12}$.($\nabla\Phi$ may be contained by the classes $\mathcal{D}_1 \oplus \mathcal{D}_2, \mathcal{D}_1 \oplus\mathcal{C}_{12},\mathcal{D}_2 \oplus\mathcal{C}_{12},\mathcal{D}_1 \oplus \mathcal{D}_2\oplus\mathcal{C}_{12}$).
\end{theorem}
\begin{proof}
Let $\nabla_\xi\xi \neq0$. By proposition [\ref{proposition3}], $i_5(\nabla\Phi)\neq 0$. So, $\nabla\Phi$ can not be in $\mathcal{D}_1$ and $\mathcal{C}_{12}$. Besides, by the  proposition [\ref{proposition1}], we have $i_{16}(\nabla\Phi)\neq0$, then  $\nabla\Phi$ can not be in $\mathcal{D}_2$.
\end{proof}

In particular, the a.c.m.s. can not belong to any subclasses of $\mathcal{D}_1$ and $\mathcal{D}_2$.

\begin{theorem} \label{theorem3}
Let $(\phi,\xi,\eta,g)$ be an a.c.m.s. obtained from a nearly parallel $G_2$ structure $\varphi$. Then, $\nabla_\xi\xi =0$ if and only if $M$ is almost K-contact.
\end{theorem}
\begin{proof}
The defining relation of almost K-contact manifolds is $\nabla_\xi \phi=0.$ Since $\varphi$ is nearly parallel, for any vector field $x$,
\begin{align*}
(\nabla_\xi \phi)(x) &= \nabla_\xi(\phi x)-\phi(\nabla_\xi x)=\nabla_\xi(\xi\times x)-\xi\times \nabla_\xi x \\
&=(\nabla_\xi\xi\times x)+(\xi\times \nabla_\xi x)-(\xi\times \nabla_\xi x)=\nabla_\xi\xi\times x,
\end{align*}
that is zero if and only if $\nabla_\xi\xi$ is zero.
\end{proof}

\begin{theorem} \label{theorem4}
 Let $(\phi,\eta,\xi,g)$ be an almost contact metric structure induced by a $G_2$ structure and $v=\sum_{i=1}^6 e_i\times\nabla_{e_i}\xi$. If $g(\xi,v)\neq0$, then $\nabla\Phi$ is not of classes $\mathcal{D}_1,\mathcal{C}_5,\mathcal{C}_7,\mathcal{C}_8,\mathcal{C}_9,\mathcal{C}_{10},\mathcal{C}_{11},\mathcal{C}_{12}$.
\end{theorem}
\begin{proof}
First to compute $i_{10}(\nabla\Phi)$, we write
$$
\begin{array}{rcl}
(\nabla_{e_i}\Phi)(e_i,\xi)&=&g(e_i,\nabla_{e_i}(\xi\times\xi))+g(\nabla_{e_i}\xi,\xi\times e_i)\\
&=&g(e_i\times\nabla_{e_i}\xi,\xi),
\end{array}
$$
and we obtain
$$i_{10}(\nabla\Phi)=\sum_{i,j=1}^6g(e_i\times\nabla_{e_i}\xi,\xi)g(e_j\times\nabla_{e_j}\xi,\xi)=g^2(v,\xi).$$
Assume that $g(\xi,v)\neq0$. Then $i_{10}(\nabla\Phi)=g(\xi,v)^2 \neq0$ and the classes $\mathcal{D}_1,\mathcal{C}_5,\mathcal{C}_7,\mathcal{C}_8,\mathcal{C}_9,\mathcal{C}_{10},\mathcal{C}_{11},\mathcal{C}_{12}$ are eliminated similar to previous proofs.
\end{proof}

\begin{corollary}
 If $g(\xi,v)\neq0$ and $div(\xi)\neq0$, then $\nabla\Phi$ is not an element of the classes $\mathcal{C}_i$, for $i=1,\cdots,12$.
 \end{corollary}

Next we give examples of a.c.m.s. induced by a parallel $G_2$ structure and a nearly parallel $G_2$ structure, respectively. The a.c.m.s. induced by the parallel $G_2$ structure is in the class $\mathcal{D}_1$, whereas that induced by the nearly parallel $G_2$ structure is almost-K-contact.

\begin{example}Let $(K, g_K)$ be a 4-dimensional K$\ddot{a}$hler manifold with an exact K$\ddot{a}$hler form $\Omega$, i.e. $\Omega=d\lambda$, where $\lambda$ is a 1-form on $K$. Consider $\mathbb{R}^3$ with coordinates $(x_1,x_2,x_3)$ and Euclidean metric $h=dx_1^2+dx_2^2+dx_3^2$. It is known that $(M=\mathbb{R}^3\times K, g=h\times g_K)$ admits a parallel $G_2$ structure
$$\varphi=dx_1\wedge dx_2\wedge dx_3+dx_1\wedge\Omega+dx_2\wedge Re\theta-dx_3\wedge Im\theta,$$
where $\theta$ is a volume form on $K$.

For all $p\in K$, there exist complex coordinates $(z_1, z_2)$ near $p$ such that
$$g_K=|dz_1|^2+|dz_2|^2, \qquad \Omega=\frac{i}{2}(dz_1\wedge d\overline{z}_1+dz_2\wedge d\overline{z}_2), \qquad \theta=dz_1\wedge dz_2$$
at $p$. Setting $z_1=x_4+ix_5$, $z_2=x_6+ix_7$, one has $$g_K=dx_4^2+\ldots+dx_7^2, \qquad \Omega=dx_4\wedge dx_5+dx_6\wedge dx_7,$$
$$ Re\theta=dx_4\wedge dx_6-dx_5\wedge dx_7, \qquad Im\theta=dx_4\wedge dx_7+dx_5\wedge dx_6$$
at $p$. Thus $g=h\times g_K=dx_1^2+\ldots+dx_7^2$ and
$$
\begin{array}{rcl}
\varphi&=&dx_1\wedge dx_2\wedge dx_3+dx_1\wedge( dx_4\wedge dx_5+dx_6\wedge dx_7)\\
&&+dx_2\wedge (dx_4\wedge dx_6-dx_5\wedge dx_7)-dx_3\wedge (dx_4\wedge dx_7+dx_5\wedge dx_6).
\end{array}
$$
(see \cite{JOYCE}).

Consider the a.c.m.s. $(\phi,\xi,\eta,g)$ on $M$ induced by the parallel $G_2$ structure $\varphi$, for $\xi=x_2\partial x_1$ and $\phi(x)=\xi\times x$. For the covariant derivative of the metric on a product manifold, see \cite{BLAIROUBINA}.
$$
\begin{array}{rcl}
(\nabla_{\partial x_2}\Phi)(\partial x_2,\partial x_3)&=&g(\partial x_2,\nabla_{\partial x_2}(x_2\partial x_1\times\partial x_3))+g(\nabla_{\partial x_2}\partial x_3, x_2\partial x_1\times\partial x_2)\\
&=&-g(\partial x_2,\nabla_{\partial x_2}(x_2\partial x_2) )\\
&=&-g(\partial x_2,\partial x_2[x_2]\partial x_2)\\
&=&-1,
\end{array}
$$
and thus the structure is not cosymplectic. We show that this structure is in the class $\mathcal{D}_1$. Since $\nabla_{\xi}\xi=x_2\partial x_1[x_2]\partial x_1=0$, we have
$$
\begin{array}{rcl}
(\nabla_{\xi}\Phi)(y,z)&=&g(y,\nabla_{\xi}(\xi\times z))+g(\nabla_{\xi}z, \xi\times y)\\
&=&g(y,(\nabla_{\xi}\xi)\times z)+g(y,\xi\times\nabla_{\xi}z)+g(\nabla_{\xi}z,\xi\times y)\\
&=&g(y,(\nabla_{\xi}\xi)\times z)\\
&=&0.
\end{array}
$$
In addition, $\nabla_x\xi=\nabla_x{x_2\partial x_1}=x[x_2]\partial x_1+x_2\nabla_x\partial x_1$ for any vector field $x$. It can be seen that $\nabla_x\partial x_1=0$ for any $x\in M=\mathbb{R}^3\times K$ and thus $\nabla_x\xi=x[x_2]\partial x_1$. Then
$$
\begin{array}{rcl}
(\nabla_{x}\Phi)(\xi,y)&=&g(\xi,\nabla_{x}(\xi\times y))+g(\nabla_{x}y, \xi\times \xi)\\
&=&g(\xi,(\nabla_{x}\xi)\times y)+g(\xi,\xi\times\nabla_{x}y)\\
&=&-g((\nabla_{x}\xi)\times \xi,y)\\
&=&-x[x_2]x_2g(\partial x_1\times\partial x_1,y)\\
&=&0.
\end{array}
$$
As a result, $\nabla\Phi\in\mathcal{D}_1$ by the definition of the space $\mathcal{D}_1$.
\end{example}

\begin{example}A Sasakian manifold is a normal contact metric manifold, or equivalently, an almost contact metric structure $(\phi,\xi,\eta,g)$ such that
$$(\nabla_x\phi)(y)=g(x,y)\xi-\eta(y)x,$$
see \cite{BLAIR}.
A 7-dimensional 3-Sasakian manifold is a Riemannian manifold $(M, g)$ equipped with three Sasakian structures $(\phi_i,\xi_i,\eta_i,g)$, $i=1,2,3$ satisfying
$$[\xi_1,\xi_2]=2\xi_3,\ \ [\xi_2,\xi_3]=2\xi_1,\ \ [\xi_3,\xi_1]=2\xi_2$$
 and
$$\phi_3\circ\phi_2=-\phi_1+\eta_2\otimes\eta_3,\ \ \phi_2\circ\phi_3=\phi_1+\eta_3\otimes\eta_2,$$
$$\phi_1\circ\phi_3=-\phi_2+\eta_3\otimes\eta_1,\ \ \phi_3\circ\phi_1=\phi_2+\eta_1\otimes\eta_3,$$
$$\phi_2\circ\phi_1=-\phi_3+\eta_1\otimes\eta_2,\ \ \phi_1\circ\phi_2=\phi_3+\eta_2\otimes\eta_1.$$

There exists a local orthonormal frame $\{e_1, \cdots, e_7\}$ such that $e_1=\xi_1$, $e_2=\xi_2$ and $e_3=\xi_3$. The corresponding coframe via the Riemannian metric
 is denoted by $\{\eta_1, \cdots, \eta_7\}$. The differentials $d\eta_i$, $i=1,2,3$ are
$$d\eta_1=-2(\eta_{23}+\eta_{45}+\eta_{67}), \ \ d\eta_2=2(\eta_{13}-\eta_{46}+\eta_{57}), \ \ d\eta_3=-2(\eta_{12}+\eta_{47}+\eta_{56}).$$
The 3-form
$$\varphi=\frac{1}{2}\eta_1\wedge d\eta_1-\frac{1}{2}\eta_2\wedge d\eta_2-\frac{1}{2}\eta_3\wedge d\eta_3$$
is a
nearly parallel $G_2$ structure (i.e. $d\varphi_i=-4\ast\varphi_i$) on $M$, constructed in \cite{FRD}. For properties and examples of Sasakian and 3-Sasakian manifolds \cite{BOYER} is a good reference.
The brackets of $e_1, \cdots, e_7$ are computed by using the differentials $d\eta_i$ and by the equation $d\eta(e_i,e_j)=-\eta([e_i,e_j])$ for frame elements $e_i, e_j$.
$$[e_1,e_2]=2e_3, \ \ \ \ [e_2,e_3]=2e_1, \ \ \ \ [e_3,e_1]=2e_2,$$
$$[e_4,e_5]=2e_1, \ \ \ \ [e_6,e_7]=2e_1, \ \ \ \ [e_4,e_6]=2e_2,$$
$$[e_5,e_7]=-2e_2, \ \ \ \ [e_4,e_7]=2e_3, \ \ \ \ [e_5,e_6]=2e_3.$$
By the Kozsul formula we obtain $\nabla_{e_i}e_i=0$ and
$$\nabla_{e_1}e_2=e_3, \nabla_{e_1}e_3=-e_2, \nabla_{e_1}e_4=-e_5, \nabla_{e_1}e_5=e_4, \nabla_{e_1}e_6=-e_7, \nabla_{e_1}e_7=e_6,$$
$$\nabla_{e_2}e_1=-e_3,\nabla_{e_2}e_3=e_1, \nabla_{e_2}e_4=-e_6, \nabla_{e_2}e_5=e_7, \nabla_{e_2}e_6=e_4, \nabla_{e_2}e_7=-e_5,$$
$$\nabla_{e_3}e_1=e_2, \nabla_{e_3}e_2=-e_1,\nabla_{e_3}e_4=-e_7, \nabla_{e_3}e_5=-e_6, \nabla_{e_3}e_6=e_5, \nabla_{e_3}e_7=e_4,$$
$$\nabla_{e_4}e_1=-e_5, \nabla_{e_4}e_2=-e_6, \nabla_{e_4}e_3=-e_7, \nabla_{e_4}e_5=e_1, \nabla_{e_4}e_6=e_2, \nabla_{e_4}e_7=e_3,$$
$$\nabla_{e_5}e_1=e_4, \nabla_{e_5}e_2=e_7, \nabla_{e_5}e_3=-e_6, \nabla_{e_5}e_4=-e_1, \nabla_{e_5}e_6=e_3, \nabla_{e_5}e_7=-e_2,$$
$$\nabla_{e_6}e_1=-e_7, \nabla_{e_6}e_2=e_4, \nabla_{e_6}e_3=e_5, \nabla_{e_6}e_4=-e_2, \nabla_{e_6}e_5=-e_3, \nabla_{e_6}e_7=e_1,$$
$$\nabla_{e_7}e_1=e_6, \nabla_{e_7}e_2=-e_5, \nabla_{e_7}e_3=e_4, \nabla_{e_7}e_4=-e_3, \nabla_{e_7}e_5=e_2, \nabla_{e_7}e_6=-e_1.$$
By the local expression of
$$
\begin{array}{rcl}
\varphi&=&\frac{1}{2}\eta_1\wedge d\eta_1-\frac{1}{2}\eta_2\wedge d\eta_2-\frac{1}{2}\eta_3\wedge d\eta_3\\
&=&\eta_{123}-\eta_{145}-\eta_{167}+\eta_{246}-\eta_{257}+\eta_{347}
+\eta_{356},
\end{array}
$$
the 2-fold vector cross products of frame elements are
$$e_1\times e_2=e_3, e_1\times e_3=-e_2, e_1\times e_4=-e_5, e_1\times e_5=e_4, e_1\times e_6=-e_7, e_1\times e_7=e_6,$$
$$e_2\times e_3=e_1, e_2\times e_4=e_6, e_2\times e_5=-e_7, e_2\times e_6=-e_4, e_2\times e_7=e_5,$$
$$e_3\times e_4=e_7, e_3\times e_5=e_6, e_3\times e_6=-e_5, e_3\times e_7=-e_4,$$
$$e_4\times e_5=-e_1, e_4\times e_6=e_2, e_4\times e_7=e_3, e_5\times e_6=e_3, e_5\times e_7=-e_2, e_6\times e_7=-e_1.$$
Consider the a.c.m.s. $(\phi,\xi,\eta,g)$ on $M$ induced by the 2-fold vector cross product of the nearly parallel $G_2$ structure $\varphi$, where
$\xi=e_1=\xi_1$, $\eta=d\eta_1$ and $\phi(x)=\xi\times x$. First, since
$$(\nabla_x\Phi)(y,z)=g(y,\nabla_x(e_1\times z))+g(\nabla_xz,e_1\times y),$$
we have $$(\nabla_{e_2}\Phi)(e_1,e_2)=1\neq 0$$
and thus the a.c.m.s. is not cosymplectic. In addition,
$$\nabla_{e_2}\Phi(\xi,e_2)=\nabla_{e_2}\Phi(e_1,e_2)=1,$$
implying that $\nabla\Phi\notin \mathcal{D}_1=\mathcal{C}_1\oplus\mathcal{C}_2\oplus\mathcal{C}_3\oplus\mathcal{C}_4.$
On the other hand
$$
\begin{array}{rcl}
(\nabla_{\xi}\Phi)(x,y)&=&g(x,\nabla_{e_1}(e_1\times y))+g(\nabla_{e_1}y, e_1\times x)\\
&=&g(x,(\nabla_{e_1}e_1)\times y)+g(x,e_1\times\nabla_{e_1}y)+g(\nabla_{e_1}y,e_1\times x)\\
&=&\varphi(e_1,\nabla_{e_1}y,x)+\varphi(e_1,x,\nabla_{e_1}y)\\
&=&0,
\end{array}$$
which gives that the a.c.m.s. is almost-K-contact, that is, an element of the class $$\mathcal{D}_1\oplus\mathcal{C}_5\oplus\mathcal{C}_6\oplus\mathcal{C}_7\oplus\mathcal{C}_8\oplus\mathcal{C}_9\oplus\mathcal{C}_{10}.$$
Moreover,
$$\delta\Phi(e_1)=-\sum_{i=2}^7(\nabla_{e_i}\Phi)(e_i,e_1)-(\nabla_{e_1}\Phi)(e_1,e_1)=-2$$
yields the a.c.m.s. is not semi-cosymplectic ($\mathcal{C}_1\oplus\mathcal{C}_2\oplus\mathcal{C}_3\oplus\mathcal{C}_7\oplus\mathcal{C}_8\oplus\mathcal{C}_9\oplus\mathcal{C}_{10}\oplus\mathcal{C}_{11}$).
The a.c.m.s. is not trans-Sasakian ($\mathcal{C}_5\oplus\mathcal{C}_6$):
$$(\nabla_{e_2}\Phi)(e_1,e_2)=1,$$
whereas $$\frac{1}{3}\{g(e_2,e_1)\eta(e_2)-g(e_2,e_2)\eta(e_1)\}=-\frac{1}{3},$$
i.e.,
the defining condition of being trans-Sasakian is not satisfied. In particular, the a.c.m.s. is not $\alpha$-Sasakian or $\alpha$-Kenmotsu. Note that we started with a Sasakian structure on a manifold, then we used the 2-fold vector cross product of the nearly parallel $G_2$ structure $\varphi$, however the induced a.c.m.s. is not Sasakian.

In fact, consider the a.c.m.s. induced by $\varphi$, where $\xi=ae_1+be_2+ce_3$ for constants $a, b, c$. Assume that this structure is Sasakian. Then we have
$$(\nabla_x\phi)(y)=g(x,y)\xi-\eta(y)x$$
for all vector fields $x, y$. For $x=e_1$, $y=e_2$, we get $b=\pm\frac{1}{3}$. If $x=e_1$, $y=e_4$, we obtain $b=0$. Thus the a.c.m.s. is not Sasakian.
Note that since $(\nabla_{e_1}\Phi)(e_1,e_2)=b$, $(\nabla_{e_1}\Phi)(e_1,e_3)=c$ and $(\nabla_{e_2}\Phi)(e_2,e_1)=a$, $\nabla\Phi\neq 0$ unless $\xi$ is zero. Thus the a.c.m.s. is not cosymplectic.

By direct calculation, it can be seen that for $\xi=ae_1+be_2+ce_3$, one has $\nabla_{\xi}\xi=0$. By Theorem [\ref{theorem3}], the a.c.m.s. is almost-K-contact.
\end{example}

\end{document}